\documentclass[a4paper,  centertags, oneside,12pt]{amsart}
\usepackage{mathrsfs}
\usepackage{appendix}
\usepackage{amssymb}
\usepackage{fancyhdr}
\usepackage{charter}
\usepackage{typearea}
\usepackage{pdfsync}
\usepackage[colorlinks,linkcolor=blue,anchorcolor=blue,citecolor=green]{hyperref}
\usepackage{amsmath,amstext,amsthm,amscd}
\usepackage{dsfont}
\usepackage{mathrsfs}
\usepackage[a4paper,top=3cm,bottom=3cm,left=2.5cm,right=2.5cm]{geometry}

\usepackage{enumitem}
\usepackage{color}
\setlist[1]{itemsep=5pt}
\newcommand{\comment}[1]{}

\makeatletter
      \def\@setcopyright{}
      \def\serieslogo@{}
      \makeatother

\newtheorem{theorem}{Theorem}[section]
\newtheorem{lemma}[theorem]{Lemma}
\newtheorem{corollary}[theorem]{Corollary}

\newtheorem{definition}[theorem]{Definition}

\newtheorem{assumption}[theorem]{Assumption}
\numberwithin{equation}{section}
\DeclareMathOperator{\dive}{div}
\begin{document}
\title{Contraction property on complex hyperbolic ball}

\author{Xiaoshan Li }
\address{School of Mathematics and Statistics, Wuhan University
, Wuhan, Hubei 430072, China}
\email{xiaoshanli@whu.edu.cn}
\author{Guicong Su}
\address{School of Mathematics and Statistics, Wuhan University of Technology, Wuhan, Hubei 430070, China}
\email{suguicong@whut.edu.cn}

%\thanks{\textbf{Keywords:}Isoperimetric inequality, contraction, Hardy, weighted Bergman }
%\thanks{\textbf{Mathematics Subject Classification (2020):} 32A10\;\textperiodcentered\,32H35}
\begin{abstract}
{We prove an isoperimetric inequalitie on the complex hyperbolic ball with  Assumption \ref{assumption}}. As an application, we prove a contraction property for the holomorphic functions in Hardy and weighted Bergman spaces on the complex hyperbolic ball with this assumption. The results can be seen as partial generalization of Kulikov's result on the complex hyperbolic plane.
\end{abstract}
\maketitle
%\tableofcontents
\section{{Introduction}}
Let $\mathbb{B}_n=\{z\in\mathbb{C}^n: \vert z\vert<1\}$ be the unit ball in $\mathbb{C}^n$.
The boundary of $\mathbb{B}_n$ will be denoted by $\mathbb{S}_n$. Let $ dv$ denote the Euclidean volume measure on $\mathbb{B}_n$, normalized so that $v(\mathbb{B}_n)=1$.
The surface measure on $\mathbb{S}_n$ will be denoted by $d\sigma$.

In this paper we work with several holomorphic function spaces defined on the unit ball $\mathbb{B}_n$. We begin with the definitions of the appropriate Hardy and weighted Bergman spaces (refer to Zhao-Zhu \cite{Z-Z} and Zhu \cite{Zhu}).
\begin{definition}
For $0<p<\infty$ the Hardy space $H^p$ consists of holomorphic functions $f$ in $\mathbb{B}_n$ such that
$$\Vert f\Vert_{H^p}^p: =\sup_{0<r<1}\int_{\mathbb{S}_n}\vert f(r\zeta)\vert^p d\sigma(\zeta)<\infty.$$
\end{definition}

We denote by $d v_g$ the hyperbolic volume measure with respect to the Bergman metric on  $\mathbb{B}_n$, that is,
$$d v_g(z)=\frac{d v(z)}{(1-\vert z\vert^2)^{n+1}}.$$

\begin{definition}
For $0<p<\infty$ and $\alpha>n$ the weighted Bergman space $A_{\alpha}^p$ consists of holomorphic functions $f$ in $\mathbb{B}_n$ such that
$$\Vert f\Vert_{A_\alpha^p}^p=\int_{\mathbb{B}_n}c_\alpha\vert f(z)\vert^p(1-\vert z\vert^2)^\alpha d v_g(z)<\infty,$$
where $c_\alpha=\frac{\Gamma(\alpha)}{n!\Gamma(\alpha-n)}.$
\end{definition}

Note that for the function $f(z)\equiv 1$ we have $\Vert f\Vert_{H^p}=\Vert f\Vert_{A_\alpha^q}=1$ for all admissible values of $p$, $q$ and $\alpha$.

An important property of these spaces is that point evaluations are continuous
in them. Specifically, for all holomorphic functions $f\in A_\alpha^p$ we have
\begin{equation}\label{berg}
\vert f(z)\vert^p(1-\vert z\vert^2)^\alpha\leq \Vert f\Vert_{A_\alpha^p}^p,
\end{equation}
and  for all holomorphic functions $f\in H^{p}$ we have
\begin{equation}\label{hardy}
\vert f(z)\vert^p(1-\vert z\vert^2)^n\leq \Vert f\Vert^p_{H^p}.
\end{equation}
In particular, when $p=nr$ and $f\in H^{nr}$ we have
$$|f(z)|^r(1-|z|^2)\leq \|f\|^r_{H^{nr}}.$$
Since polynomials are dense in all of these spaces, for each fixed function $f$ in one of these spaces, the quantities on left-hand sides of (\ref{berg}) and (\ref{hardy}) tend uniformly to $0$ as $|z|\rightarrow 1$.

On the other hand, we have
\begin{equation*}
A_\alpha^p\subset A_\beta^q, \ \ \ \frac{p}{\alpha}=\frac{q}{\beta}=r, \ \ \ p<q
\end{equation*}
and $H^{nr}$ is contained in all these spaces.
Moreover, the $H^{nr}$-norm can be evaluated as the limit of these weighted Bergman norms in the sense that for $f\in H^{nr}$ we have
$$\Vert f\Vert_{H^{nr}}=\lim_{\alpha\rightarrow n^{+}} \Vert f\Vert_{A_\alpha^{\alpha r}}.$$

For $n=1$, it was asked by Lieb and Solovej \cite{LS21} whether these embeddings are actually contractions, that is, whether the norm $\Vert f\Vert_{A_\alpha^{r\alpha}}$ is decreasing in $\alpha$. They proved that the Wehrl-type entrop conjecture for the $SU(1, 1)$ group can be deduced from such contraction. In the case of contractions from the Hardy spaces to Bergman spaces, it was asked by Pavlovic \cite{Pa14} and by Brevig, Ortega-Cerda, Seip, and Zhao \cite{BCSZ18} concerning the coefficients estimates for holomorphic functions. Recently, Kulikov \cite{Kulikov} confirmed these conjectures. A key ingredient of Kulikov's proof is the well-known isoperimetric inequality on the complex hyperbolic  plane \cite{Sch40} (see also \cite{Iz15, Osserman}). However, to our best knowledge, the isoperimetric inequalities on higher dimensional complex hyperbolic space are not known in the literature. First, based on a comparison formula for the total curvature of level sets of functions on Riemann manifolds (see \cite{MJoel, MJoel1}) we have the following isoperimetric inequality on the complex hyperbolic ball.

\begin{theorem}\label{main11}
Let $(\mathbb B_n, g)$ be the complex hyperbolic ball with $g$ the Bergman metric. Let $\Omega\Subset\mathbb B_n$ be a relatively compact subset. Then the following isoperimetric inequality holds.
\begin{equation}
{\rm per}_g(\Omega)^{2n}\geq\frac{{\rm per}(\mathbb B_n)^{2n}}{{\rm vol}(\mathbb B_n)^{2n-1}} {{\rm vol}_g(\Omega)}^{2n-1},
\end{equation}
where ${\rm per}$ stands for the perimeter, ${\rm vol}_g(\cdot)$, ${\rm per}_g(\cdot)$, ${\rm per}(\cdot)$ and ${\rm vol}(\cdot)$ are associated with Bergman metric and Euclidean metric, respectively.
\end{theorem}
It was asked by Pansu that which compact domains in complex hyperbolic space minimize boundary volume among the domains of equal volume (see \cite{Mc03}). This problem is still open now. A domain in the complex hyperbolic ball with this minimizing property is called an isoperimetric region. It is well-known that the isoperimetric region in the complex hyperbolic ball always exist, see\cite[Section 4]{Ri23}. It is very likely that the isoperimetric region in the complex hyperbolic ball is a geodesic ball. {So, we make following isoperimetric assumption first.
\begin{assumption}\label{assumption}
Assume the isoperimetric region in the hyperbolic ball is a geodesic ball.
\end{assumption}}

With the above Assumption \ref{assumption}, we obtained another isoperimetric inequality on the complex hyperbolic ball.

\begin{theorem}\label{Iso1}
Under the Assumption \ref{assumption}, let $\Omega$ be a relatively compact subset  in  the complex hyperbolic ball $(\mathbb{B}_n, g)$. Then we have
\begin{equation*}
(\mathrm{per}_g(\Omega))^2\geq 4n^2 (\mathrm{vol}_g(\Omega))^{\frac{2n-1}{n}}+4n^2 (\mathrm{vol}_g(\Omega))^2.
\end{equation*}
\end{theorem}

Making use of the isoperimetric inequality in Theorem \ref{Iso1}, we may extend Kulikov's results for Hardy and weighted Bergman spaces to higher dimensional complex hyperbolic ball $\mathbb B_n$ by proving the following
\begin{theorem}\label{increasing}
Under the Assumption \ref{assumption},
let $G:[0,\infty)\rightarrow\mathbb{R}$ be an increasing function. Then the maximum
value of
\begin{equation}
\int_{\mathbb{B}_n}
G\left(\vert f(z)\vert^r(1-\vert z\vert^2)\right)dv_g(z)
\end{equation}
is attained for $f(z)\equiv 1$, subject to the condition that $f\in H^{nr}$ and $\Vert f\Vert_{H^{nr}}=1$.
\end{theorem}

\begin{theorem}\label{convex}
Under the Assumption \ref{assumption},
let $G:[0,\infty)\rightarrow\mathbb{R}$ be a convex function. Then the maximum
value of
\begin{equation}\label{convex case}
\int_{\mathbb{B}_n}
G\left(\vert f(z)\vert^p(1-\vert z\vert^2)^{\alpha}\right)dv_g(z)
\end{equation}
is attained for $f(z)\equiv 1$, subject to the condition that $f\in A^{p}_\alpha$ and $\Vert f\Vert_{A^{p}_\alpha}=1$.
\end{theorem}

Applying these theorems to the convex and increasing function $G(t) = t^s, t>0$ with $s > 1$, we
get that all the embeddings above between Hardy and weighted Bergman spaces are contractions.

\begin{corollary}\label{cor}
Under the Assumption \ref{assumption},
for all $0<p<q<\infty$ and $n<\alpha<\beta<\infty$ with $\frac{p}{\alpha}=\frac{q}{\beta}=r$. For
all functions $f$ holomorphic in $\mathbb{B}_n$ we have
$$\Vert f\Vert_{A_\beta^q}\leq \Vert f\Vert_{A_\alpha^p}\leq \Vert f\Vert_{H^{nr}}$$
with equality for $f(z)\equiv c$ for  some constant $c$.
\end{corollary}

Kalaj \cite{Kalaj} proved a contraction property in the higher dimensional unit ball for certain class of smooth functions whose absolute values are log-subharmonic functions. However, Kalaj's class of such smooth functions does not contain the class of holomorphic functions. Recently, Dai \cite{Dai24} established some sharp inequalities for holomorphic Fock spaces in $\mathbb C^n$.

The structure of the paper is as follows. In section \ref{section2}, we prove Theorem \ref{main11} and Theorem \ref{Iso1}. In section \ref{section3}, we prove a general monotonicity theorem for the hyperbolic measure of the superlevel sets of holomorphic functions which used an ingenious method from \cite{NT22} to the complex hyperbolic ball. Finally, we deduce from it the Theorem \ref{increasing}, Theorem \ref{convex} and Corollary \ref{cor}.

\section{Isoperimetric inequality}\label{section2}
In this section, we established two isoperimetric inequalities on the complex hyperbolic ball(see Theorem \ref{main1} and Theorem \ref{Iso}) which are crucial ingredients in the proof of the monotonicity theorem for the hyperbolic measure of superlevel sets of holomorphic functions in section \ref{section3}.
\subsection{Total curvature and the isoperimetric inequality}
A Cartan-Hadamard manifold $M$ is a complete simply connected $n$-dimensional Riemannian space
with nonpositive sectional curvature. The Cartan-Hadamard conjecture states
that, if $n\geq 2$, the isoperimetric inequality
\begin{equation}\label{con}
\mathrm{per}(\Omega)^n\geq \frac{\mathrm{per}({\mathbf B}_n)^n}{\mathrm{vol}({\mathbf B}_n)^{n-1}}\mathrm{vol}(\Omega)^{n-1}
\end{equation}
holds for any bounded set $\Omega\subset M$ of finite perimeter,
where $\mathrm{per}(\cdot)$ stands for the perimeter of the associated set, and ${\mathbf B}_n$ is the unit ball
of $\mathbf{R}^n$. The equality holds if and only if $\Omega$ is isometric to the unit ball of $\mathbb{R}^n$.
The Cartan-Hadmard conjecture was first proved to be true for $n=2$ by  Weil \cite{Weil} and Beckenbach-Rad\'{o} \cite{Beckenbach}, and later only for dimensions $n=3$ by Kleiner \cite{Kleiner}, and $n=4$ by Croke \cite{Croke}. Recently, Ghomi-Spruck \cite {MJoel} developed a comparison formula for total curvature of level sets of functions on Riemann manifolds for solving this conjecture.

 The boundary $\Gamma=\partial\Omega$ of a compact subset $\Omega$ of $M$ with interior points and convex in the geodesic sense, is called a closed convex hypersurface. If $\Gamma$ is of class $\mathcal{C}^{1,1}$, then its Gauss-Kronecker curvature, or determinant of second
fundamental form, denoted by $\mathrm{GK}$ is well-defined almost everywhere. So the total curvature of $\Gamma$ can be defined by
\begin{equation*}
\mathcal{G}(\Gamma):=\int_{\Gamma} \vert\mathrm{GK}\vert d\sigma,
\end{equation*}
where $d\sigma$ is the surface area of $\Gamma$ induced by the Riemann metric on $M$.

When $n=3$, Kleiner \cite{Kleiner} showed that when the total curvature of every compact convex hypersurface $\Gamma$ with $\mathcal G(\Gamma)\geq {\rm vol}(\mathbf S^2)$,   where $\mathbf S^2$ is the unit sphere in $\mathbf R^2$, then the isoperimetric inequality holds. It was stated in the same paper that this implication should hold in all dimensions. Recently, Ghomi-Spruck \cite{MJoel} 
proved that this is indeed the case.
\begin{theorem}(\cite[Theorem 7.1]{MJoel})\label{MJoel}
Suppose that the total curvature of every compact convex  hypersurface $\Gamma$ of class $\mathcal{C}^{1,1}$ satisfies the curvature inequality
\begin{equation}\label{MJoel.1}
\mathcal{G}(\Gamma)\geq\mathrm{vol}({\mathbf S}^{n-1}),
\end{equation}
where $\mathbf S^{n-1}$ denotes the unit sphere in $\mathbf{R}^n$ and ${\rm vol}(\cdot)$ stands for the volume of the associated set with respect to the Euclidean metric. Then all bounded subsets of $M$ satisfy the isoperimetric inequality \eqref{con}, with equality only for
Euclidean balls.
\end{theorem}

Thus, in order to obtain the isoperimetric inequality, we only need to verify the lower bound of the total curvature. However, the computation of
the total curvature seems to be impossible. Hence, we need some basic facts and  a comparison formula to help us to estimate the lower bound.

\begin{lemma}(\cite[Lemma 4.2]{MJoel1})\label{geodesic}
Let $S_\rho$ be a geodesic sphere of radius $\rho$ centered at a point in a $n$-dimensional Riemannian manifold. As $\rho\rightarrow 0$ we have
$$\mathcal{G}(S_\rho)\rightarrow\mathrm{vol}({\mathbf S}^{n-1}).$$
\end{lemma}

In the following, we introduce a comparison formula which plays an essential role to verify the inequality \eqref{MJoel.1} in Theorem \ref{MJoel}.
\begin{theorem}(\cite[Theorem 4.7]{MJoel})\label{comparison}
Let $\Gamma$ and $\gamma$ be closed $\mathcal{C}^{1,1}$ hypersurfaces in a Riemannian $n$-manifold $M$ bounding domains
$\Omega$ and $D$, respectively, with $\overline{D}\subset\Omega$. Suppose there exists a
 $\mathcal{C}^{1,1}$ function $u$ on $\overline{\Omega\setminus D}$ with $\nabla u\neq 0$, which is
constant on $\Gamma$ and $\gamma$. Let $\kappa^u:=\left(\kappa^u_1,\ldots,\kappa^u_{n-1}\right)$ be the principal curvatures of level sets of $u$ with respect to
$E_n:=\frac{\nabla u}{\vert\nabla u\vert}$, and let $E_1, \ldots, E_{n-1}$ be the corresponding principal directions. Then, we have
\begin{equation*}
\begin{aligned}
\mathcal{G}(\Gamma)-\mathcal{G}(\gamma)=&-\int_{\Omega\setminus D}\sum\kappa^u_{i_1}\cdots\kappa^u_{i_{n-2}}K_{i_{n-1},n}d\mu\\
&+\frac{1}{\vert\nabla u\vert}\int_{\Omega\setminus D}\sum\kappa^u_{i_1}\cdots\kappa^u_{i_{n-3}}\vert\nabla u\vert_{i_{n-2}} R_{i_{n-1},i_{n-2},i_{n-1},n} d\mu,
\end{aligned}
\end{equation*}
where $d\mu$ denotes the $n$-dimensional Riemannian volume measure on $M$, $\vert\nabla u\vert_i:=\nabla_{E_i}\vert\nabla u\vert$, $R_{ijkl}=R(E_i,E_j, E_k, E_l)$
are components of the Riemann curvature tensor of $M$, $K_{ij}=R_{ijij}$ is the sectional curvature, and the summations take place over distinct values of
$1\leq i_1,\dots,i_{n-1}\leq n-1$, with $i_1<\cdots<i_{n-2}$ in the first term and $i_1<\cdots<i_{n-3}$ in the second term.
\end{theorem}

\subsection{Isoperimetric inequality on complex hyperbolic ball}

Let  $M$ be the unit ball $\mathbb{B}_n$ in $\mathbb{C}^n$ with Bergman metric $g$, and let $\Gamma$ and $\gamma$ be closed convex $\mathcal{C}^{1,1}$ hypersurfaces in $\mathbb{B}_n$ bounding domains
$\Omega$ and $D$, respectively, with $\overline{D}\subset\Omega$.  We can find $u$ which is convex in
$\mathbb B_n$ and satisfies the assumptions as in Theorem \ref{comparison} (see \cite[Lem.1]{B02}).

As we know,  $(\mathbb{B}_n, g)$ has
negative constant holomorphic sectional curvature. Hence, its Riemann curvature tensor with respect to the local coordinates $z=(z_1, \cdots, z_n)$ centered at a given point, where $z_j=x_j+i x_{j+n}, 1\leq j\leq n$ and $g(\frac{\partial}{\partial x_i}, \frac{\partial}{\partial x_j})(0)=\delta_{ij}, 1\leq i, j\leq 2n$ has the following form.
\begin{equation}\label{curvature}
\begin{aligned}
R_{ijkl}(0)&=\delta_{il}\delta_{jk}-\delta_{ik}\delta_{jl},\\
R_{i\bar{j}k\bar{l}}(0)&=-\delta_{il}\delta_{\bar{j}\bar{k}}-\delta_{ik}\delta_{\bar{j}\bar{l}}-2\delta_{ij}\delta_{kl},\\
R_{i\bar{j}kl}(0)&=0, \ \ \ R_{ijk\bar{l}}(0)=0,
\end{aligned}
\end{equation}
where  the index $\bar{l}$ denotes $l+n$ and $R_{ijkl}=R(\frac{\partial}{\partial x_i}, \frac{\partial}{\partial x_j}, \frac{\partial}{\partial x_k}, \frac{\partial}{\partial x_l}), 1\leq i, j, k, l\leq 2n$. Moreover, we can assume that $\frac{\partial}{\partial x_{2n}}|_0=\frac{\nabla_g u}{|\nabla_g u|_g}(0)$ and $\frac{\partial}{\partial x_j}|_0, 1\leq j\leq 2n-1$ are corresponding principal directions at $0$. From (\ref{curvature}) one has
$$R_{i_{2n-1}, i_{2n-2}, i_{2n-1}, 2n}(0)=0.$$
Hence, Theorem \ref{comparison} yields
\begin{equation*}
\begin{aligned}
\mathcal{G}(\Gamma)-\mathcal{G}(\gamma)&=-\int_{\Omega\setminus D}\sum\kappa^u_{i_1}\cdots\kappa^u_{i_{2n-2}}K_{i_{2n-1}, 2n}d\mu\\
&\geq -K \int_{\Omega\setminus D}\sum\kappa^u_{i_1}\cdots\kappa^u_{i_{2n-2}}d\mu\geq 0,
\end{aligned}
\end{equation*}
where $K<0$ is the upper bound for sectional curvatures of $\mathbb{B}_n$ with respect to the Bergman metric $g$, and each $\kappa^u_j\geq 0$ since $u$ is convex.

Finally, letting $\gamma$ be a sequence of geodesic balls with  radius approaching $0$, Lemma \ref{geodesic} shows that
$$\mathcal{G}(\Gamma)\geq\mathrm{vol}({\mathbf S}^{2n-1}).$$
Therefore, by Theorem \ref{MJoel} we can obtain the following isoperimetric inequality on the complex hyperbolic ball.
\begin{theorem}\label{main1}
Let $(\mathbb B_n, g)$ be the complex hyperbolic ball with $g$ the Bergman metric. Let $\Omega\Subset\mathbb B_n$ be a relatively compact subset. Then the following isoperimetric inequality holds.
\begin{equation}
{\rm per}_g(\Omega)^{2n}\geq\frac{{\rm per}(\mathbb B_n)^{2n}}{{\rm vol}(\mathbb B_n)^{2n-1}} {{\rm vol}_g(\Omega)}^{2n-1},
\end{equation}
where ${\rm vol}_g(\cdot)$, ${\rm per}_g(\cdot)$ and ${\rm per}(\cdot)$, ${\rm vol}(\cdot)$ are associated with Bergman metric and Euclidean metric, respectively.

\end{theorem}

\subsection{Rearrangement function and P\'{o}lya-Szeg\H{o} inequality}\label{section}

We start this section by recalling the Coarea formula for functions on a Riemannian manifold.
\begin{theorem}(\cite[pp. 302-303]{Chavel})\label{coarea}
Let $\Omega$ be a relatively compact domain
in the $n$-dimensional Riemannian manifold $(M,h)$,
and $f:\overline\Omega\rightarrow [0, +\infty)$ a function in
$\mathcal{C}^{0}(\overline\Omega)\cap \mathcal{C}^{\infty}(\Omega)$, with $f|_{\partial\Omega}=0$. For any regular value $t$ of
$ f$, we let
\begin{equation*}
\Gamma(t)= f^{-1}(t), \ \ \ \ \ \ A_t=A(\Gamma(t)),
\end{equation*}
where $A(\cdot)$ is the area of the associated set. We denote by $d A_t$ the $(n-1)$-dimensional Riemannian measure on $\Gamma(t)$. Then,
\begin{equation*}
dV|_{\Gamma(t)}=\frac{d A_t dt}{\vert\nabla f\vert_h},
\end{equation*}
and for any function $\phi\in L^1(\Omega)$, we have
\begin{equation*}
\int_\Omega \phi\vert\nabla f\vert_h dV=\int_0^\infty dt\int_{\Gamma(t)}\phi dA_t.
\end{equation*}
\end{theorem}

In the following, we will recall some technique from symmetrization arguments on Riemannian manifolds,
refer to Druet-Hebey \cite{D-H}, Druet-Hebey-Vaugon \cite{D-H-V}, Hebey \cite{He}, and Ni \cite{Ni}.

Let $(\mathbb B_n, g)$ be the complex hyperbolic ball with the Bergman metric $g$. Let $u: \mathbb{B}_n\rightarrow[0, +\infty)$ be a measurable function such that
\begin{equation*}
\mathrm{vol}_g\left(\{z\in\mathbb{B}_n: u(z)>t\}\right)=\int_{\{z\in\mathbb{B}_n: ~u(z)>t\}} d v_g<\infty, \ \ \ \forall t>0.
\end{equation*}
For such a function $u$, its distribution function with respect to the hyperbolic measure, denoted by $\mu$ is defined by
\begin{equation*}
\mu(t)=\mathrm{vol}_g\left(\{z\in\mathbb{B}_n: u(z)>t\}\right).
\end{equation*}
The function $[0,\infty)\ni t\mapsto\mu(t)$ is non-increasing and right continuous. We will define a hyperbolic symmetric decreasing rearrangement function associated to $u$. For reference of rearrangement functions, we refer the readers to Baernstein \cite{Ba1}.

First, we associate such  function $u$ its decreasing rearrangement function $u^*: (0, +\infty)\rightarrow(0, +\infty)$
which is defined by
\begin{equation}\label{rearrangement1}
u^*(t)=\sup\{s>0: \mu(s)>t\}, \forall~ t>0.
\end{equation}
We now define a radially symmetric decreasing function $u^\sharp_g$ of $u$ by
\begin{equation}\label{rearrangement2}
u^\sharp_g(z)=u^*\left(\mathrm{vol}_g(B_g(0,\rho(z)))\right), \ \ \ \ z\in\mathbb{B}_n,
\end{equation}
where $\rho(z)=\frac{1}{2}\ln\frac{1+\vert z\vert}{1-\vert z\vert}$ denotes the geodesic
distance from $z$ to $0$, and $B_g(0,r)$ denotes the open geodesic ball center at $0$ of radius $r$ in $\mathbb{B}_n$.
Then $u^\sharp_g(z)$ is the hyperbolic symmetric decreasing rearrangement function of $u$  on the unit ball with respect to the hyperbolic measure. Moreover, the hyperbolic symmetric decreasing rearrangement $u^{\sharp}_g$ and $u$ have the same distribution, that is, $${\rm vol}_g(\{z\in\mathbb B_n: u^{\sharp}_g(z)>t\})={\rm vol}_g(\{z\in\mathbb B_n: u(z)>t\}), \forall t>0.$$

By using the classical Morse theory and density arguments, in the following we will assume $u: {\mathbb B_n}\rightarrow [0, \infty)$ is a continuous function having compact support $S\subset\mathbb B_n$, where $S$ is smooth enough, $u$ being of class $C^2$-smooth in $S$ and having only non-degenerate critical points in $S$.

By the isoperimetric inequality established in Theorem \ref{main1}, we have the following P\'{o}lya-Szeg\H{o} type inequalities on the complex hyperbolic ball under the Assumption \ref{assumption}.
\begin{theorem}\label{PS}
Let $u:\mathbb{B}_n\rightarrow[0,+\infty)$ be a non-zero function with the above properties. Let $u^\sharp_g$ be its rearrangement function on $\mathbb{B}_n$. Then we have

(i) Volume preservation:
\begin{equation*}
\mathrm{vol}_{g}\left({\rm supp} \ u\right)=\mathrm{vol}_{g}\left({\rm supp} \ u^\sharp_g\right).
\end{equation*}

(ii) Norm preservation: for every $q\in(0,\infty]$,
\begin{equation*}
\Vert u\Vert_{L^{q}(\mathbb{B}_n)}=\Vert u^\sharp_g\Vert_{L^{q}(\mathbb{B}_n)}
\end{equation*}

(iii) P\'{o}lya-Szeg\H{o} inequality: Under the Assumption \ref{assumption}, for every $p\in(1, \infty)$,
$$\int_{\mathbb{B}_n}\vert\nabla_g\ u^\sharp_g\vert_g^p dv_g(z)\leq\int_{\mathbb{B}_n}\vert\nabla_g\ u\vert_g^p dv_g(z).$$
\end{theorem}

\begin{proof}[Proof]
By the assumption of $u$, it is clear that  $u_g^{\sharp}$ is a Lipschitz functions with compact support and by
by definition,  $u$ and $u^\sharp_g$ have the same distribution function with respect to the hyperbolic measure. Thus, one has
\begin{equation*}
\begin{aligned}
&\mathrm{vol}_{g}\left({\rm supp} \ u\right)=\mathrm{vol}_{g}\left({\rm supp}\ u^\sharp_g\right)\\
&\Vert u\Vert_{L^{\infty}(\mathbb{B}_n)}=\Vert u^\sharp_g\Vert_{L^{\infty}(\mathbb{B}_n)}
\end{aligned}
\end{equation*}
For $q\in(0,\infty)$, by the Fubini's theorem it easily follows that
\begin{equation*}
\begin{aligned}
\Vert u\Vert^q_{L^{q}(\mathbb{B}_n)}&=\int_{\mathbb{B}_n} u^q dv_g\\
&=\int_0^\infty \mathrm{vol}_{g}\left(\{z\in\mathbb{B}_n: u(z)>t^{\frac{1}{q}}\}\right) dt\\
&=\int_0^\infty \mathrm{vol}_{g}\left(\{z\in\mathbb{B}_n: u^\sharp_g (z)>t^{\frac{1}{q}}\}\right) dt\\
&=\int_{\mathbb{B}_n}(u^\sharp_g)^q dv_g(z)\\
&=\Vert u^\sharp_g\Vert_{L^{q}(\mathbb{B}_n)}^q.
\end{aligned}
\end{equation*}
Thus, we get conclusion of (i) and (ii).

To prove (iii), we will follow the arguments from Kleiner \cite{Kleiner}, Ni \cite{Ni} and Farkas \cite{FC}.  For every $0<t<\Vert u\Vert_{L^{\infty}(\mathbb{B}_n)}$, we consider the level sets
\begin{equation*}
\Gamma_t=u^{-1}(t)\subset S\subset\mathbb{B}_n,\ \ \ \ \ \Gamma_t^{\sharp}=(u^\sharp_g)^{-1}(t)\subset S\subset\mathbb{B}_{n}
\end{equation*}
which are the boundaries of the sets $\{z\in\mathbb{B}_n: u(z)>t\}$ and $\{z\in\mathbb{B}_n: u^\sharp_g(z)>t\}$, respectively.
Since $u^\sharp_g$ is radially symmetric, the set $\Gamma_t^{\sharp}$ is an $(2n-1)$-dimensional sphere for every $0<t<\Vert u\Vert_{L^{\infty}(\mathbb{B}_n)}$.
Set
\begin{equation*}
\begin{aligned}
\mu(t)&=\mathrm{vol}_{g}\left(\{z\in\mathbb{B}_n: u(z)>t\}\right)\\
&=\mathrm{vol}_{g}\left(\{z\in\mathbb{B}_n: u^\sharp_g(z)>t\}\right).
\end{aligned}
\end{equation*}
By the Coarea formula in Theorem \ref{coarea}, we have
\begin{equation*}
-\mu'(t)=\int_{\Gamma_t}\vert\nabla_g \ u\vert_g^{-1}d\sigma_g=\int_{\Gamma^{\sharp}_t}\vert\nabla_g\ u^\sharp_g\vert_g^{-1}d\sigma_g, ~{\rm a.e.} ~t\in (0, \max u)
\end{equation*}
where $d\sigma_g$  denotes the natural $(2n-1)$-dimensional surface area
element induced by the Bergman metric.
Since $\vert\nabla_g\ u^{\sharp}\vert_g$ is constant on the sphere $\Gamma^{\sharp}_t$, it turns out that
\begin{equation*}
-\mu'(t)=\frac{\mathrm{Area}_{g}(\Gamma^{\sharp}_t)}{\vert\nabla_g\ u^\sharp_g(z)\vert_g}, \ \ \ z\in\Gamma^{\sharp}_t.
\end{equation*}
Then we have
\begin{equation*}
\mathrm{Area}_{g}(\Gamma_t)=\int_{\Gamma_t} d\sigma_g\leq(-\mu'(t))^{\frac{p-1}{p}}\left(\int_{\Gamma_t}\vert\nabla_g\ u\vert_g^{p-1}d\sigma_g\right)^{\frac{1}{p}}.
\end{equation*}
Thus, we have
\begin{equation*}
\int_{\mathbb{B}_n}\vert\nabla_g\ u^\sharp_g\vert_g^p dv_g(z)=\int_0^\infty dt\int_{\Gamma^{\sharp}_t}\vert\nabla\ u^\sharp_g\vert_g^{p-1} d\sigma_g=\int_0^\infty\frac{\left(\mathrm{Area}_{g}(\Gamma^{\sharp}_t)\right)^p}{(-\mu'(t))^{p-1}}dt.
\end{equation*}
{By the isoperimetric Assumption 
\ref{assumption}, it follows that 
\begin{equation}{\mathrm {Area}}(\Gamma_t^{\sharp})\leq{\mathrm {Area}}(\Gamma_t).\end{equation}}
Thus, we have
\begin{equation*}
\begin{aligned}
\int_{\mathbb{B}_n}\vert\nabla_g\ u^\sharp_g\vert_g^p dv_g(z)
&{\leq\int_0^\infty\frac{\left(\mathrm{Area}_{g}(\Gamma_t)\right)^p}{(-\mu'(t))^{p-1}}dt}
\leq\int_0^\infty dt\int_{\Gamma_t}\vert\nabla_g\ u\vert_g^{p-1}d\sigma_g\\
&=\int_{\mathbb{B}_n}\vert\nabla_g\ u\vert_g^p dv_g(z).
\end{aligned}
\end{equation*}
\end{proof}

\subsection{Sobolev type inequalities on complex hyperbolic ball}
In this section, we will give a refinement of isoperimetric inequality on complex hyperbolic ball,  which is an improvement of Theorem \ref{main1}.

First, we  define another symmetric
decreasing function $u^\sharp_e$ on $\mathbb{C}^n$ by
\begin{equation*}
u^\sharp_e(z)=u^{*}( \Vert z\Vert^{2n}), \ \ \ \ z\in\mathbb{C}^n,
\end{equation*}
and $u^\sharp_e$ is the Euclidean rearrangement function of $u$ in $\mathbb{C}^n$, where $u^*$ is the decreasing rearrangement of $u$ with respect to the hyperbolic measure.

Motivated by Nguyen \cite{Nguyen}, we compare the two terms
$ \Vert\nabla_gu^\sharp_g\Vert^p_{L^{p}(\mathbb{B}_n)}$ and $\Vert\nabla u^\sharp_e\Vert^p_{L^{p}(\mathbb{C}^n)}$ and establish several Sobolev inequalities on the complex hyperbolic space.
\begin{theorem}\label{sobolev}
Under the Assumption \ref{assumption}, 
let $n\geq 1$ and $p\in[1,\infty)$. Then for any function $u\in W^{1,p}(\mathbb{B}_n)$, the following inequalities
holds.\\
(I) If $p=1$, then
\begin{equation}\label{p=1}
\left(\int_{\mathbb{B}_n}\vert\nabla_g\ u\vert_g dv_g\right)^2\geq
4n^2\left(\int_{\mathbb{B}_n}\vert u\vert^{\frac{2n}{2n-1}} dv_g\right)^{\frac{2n-1}{n}}+
4n^2\left(\int_{\mathbb{B}_n}\vert u\vert dv_g\right)^2.
\end{equation}
(II) If $1<p<2$ ,then
\begin{equation*}
\int_{\mathbb{B}_n}\vert\nabla_g\ u\vert_g^p dv_g\geq
\left(S(2n,p)^2\Vert u \Vert^2_{L^{\frac{2np}{2n-p}}(\mathbb{B}_n)}+
\left(\frac{2n}{p}\right)^2\Vert u\Vert^2_{L^{p}(\mathbb{B}_n)}\right)^{\frac{p}{2}}.
\end{equation*}
(III) If $n\geq 2$ and $2\leq p< 2n$, then
\begin{equation*}
\int_{\mathbb{B}_n}\vert\nabla_g\ u\vert_g^p dv_g\geq
S(2n,p)^p\left(\int_{\mathbb{B}_n}\vert u\vert^{\frac{2np}{2n-p}} dv_g\right)^{\frac{2n-p}{2n}}
+\left(\frac{2n}{p}\right)^p\int_{\mathbb{B}_n} \vert u\vert^p d v_g.
\end{equation*}
(IV) If $2n< p<\infty$, then
$$\frac{1}{2}B\left(n-\frac{(2n-1)p}{2(p-1)}, \frac{n}{p-1}\right)
\left(\int_{\mathbb{B}_n}\vert\nabla_g\ u\vert^p_g dv_g\right)^{\frac{1}{p}}\geq\sup_{z\in\mathbb{B}_n}\vert u(z)\vert,$$
where $B(P,Q)$ is the Beta function.
\end{theorem}

From the Sobolev inequality \eqref{p=1} (refer to Osserman \cite{Osserman}) we can obtain another isoperimetric inequality on complex hyperbolic ball.
\begin{theorem}\label{Iso}
Under the Assumption \ref{assumption},
let $D$ be a relatively compact subset  in  the complex hyperbolic ball $(\mathbb{B}_n, g)$. Then we have
\begin{equation*}
(\mathrm{per}_g(D))^2\geq 4n^2 (\mathrm{vol}_g(D))^{\frac{2n-1}{n}}+4n^2 (\mathrm{vol}_g(D))^2.
\end{equation*}
The equality holds if and only if $D$ is a geodesic ball in $(\mathbb B_n ,g)$.
\end{theorem}

Before the proof of Theorem \ref{sobolev}, we first recall the following two lemmas about the weighted Hardy inequality and the sharp Sobolev inequality in $\mathbb R^{n}$.
\begin{lemma}(\cite[Theorem 330]{HLP})
Assume that $p>1$ and $\varepsilon>p-1$, and then for any measurable non-negative function $f$, we have
\begin{equation*}
\left(\frac{p}{\varepsilon+1-p}\right)^p\int_0^\infty f^p(x)x^\varepsilon dx\geq \int_0^\infty\left(\frac{1}{x}\int_0^\infty f(t)dt\right)^p x^\varepsilon dx,
\end{equation*}
and the constant $\left(\frac{p}{\varepsilon+1-p}\right)^p$ is sharp.
\end{lemma}
\begin{lemma}[Aubin \cite{Aubin} and Talenti \cite{Talenti}]
For $1<p<n$ and $p^*=\frac{np}{n-p}$, we have
$$\Vert\nabla u\Vert_{L^p(\mathbb{R}^n)}\geq S(n,p)\Vert u\Vert_{L^{p*}(\mathbb{R}^n)}, \ \ \ u\in C_0^\infty(\mathbb{R}^n),$$
and the sharp constant $S(n,p)$  is given by
\begin{equation*}
S(n,p)=\left(\frac{1}{n}\left(\frac{n(p-1)}{n-p}\right)^{1-\frac{1}{p}}\left(\frac{\Gamma(n)}{\Gamma\left(\frac{n}{p}\right)\Gamma\left(n+1-\frac{n}{p}\right)}\right)^{\frac{1}{n}}
\right)^{-1}.
\end{equation*}
\end{lemma}

\begin{proof}[Proof of Theorem \ref{sobolev}]
By the classical Morse inequality  and density property we can assume that $u$ is a function with above properties (i.e., it is continuous with compact $S\subset\mathbb B_n$, $S$ being smooth enough and $u$ is of class $C^2$-smooth with only non-degenerate critical points in $S$). It is clear that $u^\ast(t)$ is an absolutely continuous function on $(0, \infty)$.
By a straightforward computation, we have
\begin{equation}\label{E1}
\begin{aligned}
\int_{\mathbb{C}^n}\vert\nabla u^\sharp_e\vert^pdv(z)&=2n\int_0^\infty\left\vert(u^*)'( \Vert z\Vert^{2n})\right\vert^p
(2n r^{2n-1})^p r^{2n-1} dr\\
&=(2n)^p\int_0^\infty\vert(u^*)'(s)\vert^p s^{\frac{(2n-1)p}{2n}}ds
\end{aligned}
\end{equation}
On the other hand, we have
\begin{equation*}
\begin{aligned}
\mathrm{vol_g}(B_g(0,\rho(z)))&=\int_{\vert w\vert<\tanh \rho(z)}dv_g(w)\\
&=\int_{\vert w\vert<\tanh \rho(z)}\frac{1}{(1-\vert w\vert^2)^{n+1}}dv(w)\\
&=2n\int_0^{\tanh\rho(z)}\frac{r^{2n-1}}{(1-r^2)^{n+1}}dr\\
&=2n\int_0^{\rho(z)}(\tanh s)^{2n-1}(\cosh^2 s)^{n+1}\frac{1}{\cosh^2 s} ds\\
&=\big(\sinh\rho(z)\big)^{2n}.
\end{aligned}
\end{equation*}
Then the gradient of $\mathrm{vol_g}(B_g(0,\rho(z)))$ is  given by
\begin{equation*}
\nabla_g \mathrm{vol_g}(B_g(0,\rho(z)))=2n\big(\sinh\rho(z)\big)^{2n-1}\cosh\rho(z)\nabla_g\rho(z).
\end{equation*}
Since $\vert\nabla_g\rho(z)\vert_g=1$ for $z\neq 0$,  we deduce that
\begin{equation*}
\begin{split}
&\int_{\mathbb{B}_n}\vert\nabla_gu^\sharp_g(z)\vert^p_gdv_g(z)\\
&=\int_{\mathbb{B}_n}\vert(u^*)'({\rm vol}_g(B_g(0, \rho(z))))\vert^p\left(
2n\big(\sinh\rho(z)\big)^{2n-1}\cosh\rho(z)\right)^p dv_g(z)\\
&=2n\int_0^\infty\vert(u^*)'({\rm vol}_g(B_g(0, \rho(z))))\vert^p
\left(2n
\big(\sinh t\big)^{2n-1}{\cosh t}\right)^p(\sinh t)^{2n-1}\cosh t dt.
\end{split}
\end{equation*}
Making the change of variable $s=\mathrm{vol_g}(B_g(0,t)))=\left(\sinh t\right)^{2n}$,  we have
$$ds= 2n(\sinh t)^{2n-1}{\cosh t}dt.$$
Thus,
\begin{equation}\label{E2}
\begin{aligned}
\int_{\mathbb{B}_n}\vert\nabla_gu^\sharp_g(z)\vert^p_gdv_g(z)
=&
(2n)^p\int_0^\infty\vert(u^*)'(s)\vert^p s^{\frac{(2n-1)p}{2n}}\left(1+s^{\frac{1}{n}}\right)^{\frac{p}{2}}ds\\
=&\Vert\nabla u^\sharp_e\Vert^p+(2n)^p\int_0^\infty\vert(u^*)'(s)\vert^p k_{n,p}(s)ds,
\end{aligned}
\end{equation}
where
\begin{equation*}
k_{n,p}(s)=s^{\frac{(2n-1)p}{2n}}\left((1+s^{\frac{1}{n}})^{\frac{p}{2}}-1\right).
\end{equation*}
It is easy to check that for $\alpha\in(0,1)$, the following holds
\begin{equation*}
(a+b)^\alpha=\sup_{t\in[0,1]}(t^{1-\alpha}a^\alpha+(1-t)^{1-\alpha}b^\alpha).
\end{equation*}
Hence, for $1< p<2$,  we have
\begin{equation*}
\begin{aligned}
&\int_{\mathbb{B}_n}\vert\nabla_gu^\sharp_g(z)\vert^p_gdv_g(z)\\
=&
(2n)^p\int_0^\infty\vert(u^*)'(s)\vert^p s^{\frac{(2n-1)p}{2n}}\left(1+s^{\frac{1}{n}}\right)^{\frac{p}{2}}ds\\
\geq&\  t^{1-\frac{p}{2}}\left((2n)^p\int_0^\infty\vert(u^*)'(s)\vert^p s^{\frac{(2n-1)p}{2n}}ds\right)
+ \ (1-t)^{1-\frac{p}{2}}\left((2n)^p\int_0^\infty\vert(u^*)'(s)\vert^p s^pds\right)\\
\geq& \ t^{1-\frac{p}{2}}\left(\Vert\nabla u^\sharp_e\Vert^p_{L^p(\mathbb{C}^n)}\right) +(1-t)^{1-\frac{p}{2}}\left(\left(\frac{2n}{p}\right)^p
\int_0^\infty\vert u^*(s)\vert^pds\right)\\
\geq&  \ t^{1-\frac{p}{2}}\left(S(2n,p)^{p}\Vert u^\sharp_e\Vert^p_{L^{\frac{2np}{2n-p}}(\mathbb{C}^n)}\right)+(1-t)^{1-\frac{p}{2}}\left(\left(\frac{2n}{p}\right)^p\Vert u^\sharp_g\Vert^p_{L^{p}(\mathbb{B}_n)}\right).
\end{aligned}
\end{equation*}
Here, we used a weighted Hardy inequality  with $\varepsilon=p$ for the first inequality, and the sharp Sobolev inequality in $\mathbb{R}^{2n}$ for the second inequality. 

Notice that $$\Vert u \Vert_{L^{p}(\mathbb{B}_n)}=\Vert u^\sharp_g\Vert_{L^{p}(\mathbb{B}_n)}, ~~
\Vert u \Vert_{L^{\frac{2np}{2n-p}}(\mathbb{B}_n)}=\Vert u^\sharp_e\Vert_{L^{\frac{2np}{2n-p}}(\mathbb{C}^n)}.$$ Thus, it follows that for every $t\in[0,1]$, we have
\begin{equation*}
\int_{\mathbb{B}_n}\vert\nabla_g\ u\vert_g^p dv_g\geq
t^{1-\frac{p}{2}}\left[S(2n,p)^{p}\Vert u \Vert^p_{L^{\frac{2np}{2n-p}}(\mathbb{B}_n)}\right]+
(1-t)^{1-\frac{p}{2}}\left[\left(\frac{2n}{p}\right)^p\Vert u\Vert^p_{L^{p}(\mathbb{B}_n)}\right].
\end{equation*}
Therefore, we deduce that
\begin{equation*}
\int_{\mathbb{B}_n}\vert\nabla_g\ u\vert_g^p dv_g\geq
\left[S(2n,p)^2\Vert u \Vert^2_{L^{\frac{2np}{2n-p}}(\mathbb{B}_n)}+
\left(\frac{2n}{p}\right)^2\Vert u\Vert^2_{L^{p}(\mathbb{B}_n)}\right]^{\frac{p}{2}}.
\end{equation*}
Letting $p\rightarrow 1^+$, we conclude that
\begin{equation*}
\left(\int_{\mathbb{B}_n}\vert\nabla_g\ u\vert_g dv_g\right)^2\geq
(2n)^2\left(\int_{\mathbb{B}_n}\vert u\vert^{\frac{2n}{2n-1}} dv_g\right)^{\frac{2n-1}{n}}+
(2n)^2\left(\int_{\mathbb{B}_n}\vert u\vert dv_g\right)^2.
\end{equation*}
Assume that $n\geq 2$. For $2\leq p< 2n$, we have
$$k_{n,p}(t)\geq t^p.$$
Thus,
\begin{equation*}
\begin{aligned}
\int_{\mathbb{B}_n}\vert\nabla_gu^\sharp_g(z)\vert^p_gdv_g(z)&\geq
\Vert\nabla u^\sharp_e\Vert^p+(2n)^p\int_0^\infty\vert(u^*)'(s)\vert^p s^pds\\
&\geq \Vert\nabla u^\sharp_e\Vert^p+\left(\frac{2n}{p}\right)^p
\int_0^\infty\vert u^*(s)\vert^pds\\
&\geq S(2n,p)^{p}\Vert u^\sharp_e\Vert^p_{L^{\frac{2np}{2n-p}}(\mathbb{C}^n)}+
\left(\frac{2n}{p}\right)^p\Vert u^\sharp_g\Vert^p_{L^{p}(\mathbb{B}_n)}.
\end{aligned}
\end{equation*}
Hence,  we deduce that
\begin{equation*}
\int_{\mathbb{B}_n}\vert\nabla_g\ u\vert_g^p dv_g\geq
S(2n,p)^p\left(\int_{\mathbb{B}_n}\vert u\vert^{\frac{2np}{2n-p}} dv_g\right)^{\frac{2n-p}{2n}}
+\left(\frac{2n}{p}\right)^p\int_{\mathbb{B}_n} \vert u\vert^p d v_g.
\end{equation*}
For $2n< p<\infty$, set
$$\ell(s)=s^{\frac{2n-1}{2n}}\left(1+s^{\frac{1}{n}}\right)^{\frac{1}{2}}.$$
It is easy to check that
$$\int_0^{\infty}\ell(s)^{-\frac{p}{p-1}} ds <\infty.$$
Moreover, we have
\begin{equation*}
    \begin{aligned}
\int_0^{\infty}\ell(s)^{-\frac{p}{p-1}} ds&=n\int_0^\infty        u^{n-1-\frac{(2n-1)p}{2(p-1)}}(1+u)^{-\frac{p}{2(p-1)}} du\\
&=nB\left(n-\frac{(2n-1)p}{2(p-1)}, \frac{n}{p-1}\right),
    \end{aligned}
\end{equation*}
where $B(P,Q)$ is the Beta function.

As we know, the rearrangement function
$u^{*}$ is decreasing, and $\lim_{t\rightarrow\infty} u^{*}(t)=0.$
Hence, we deduce that
$$u^*(t)=-\int_t^\infty
(u^*)'(x)dx=-\int_t^\infty (u^*)'(x) \ell(x) \ell^{-1}(x)dx.$$

Therefore, we have 
\begin{equation*}
    \begin{aligned}
\sup_{z\in\mathbb{B}_n}\vert u(z)\vert= u^*(0)&\leq\left(\int_0^\infty\vert (u^*)'(s)\vert^p\ell(s)^pds\right)^{\frac{1}{p}}\left(\int_0^{\infty}\ell(s)^{-\frac{p}{p-1}} ds\right)^{\frac{p-1}{p}}\\
   &=\frac{1}{2n}\left(\int_{\mathbb{B}_n}\vert\nabla_gu^\sharp_g(z)\vert^p_gdv_g(z)\right)^{\frac{1}{p}}\left(\int_0^{\infty}\ell(s)^{-\frac{p}{p-1}} ds\right)^{\frac{p-1}{p}} \  \ \mathrm{(By \ \  \eqref{E2})}\\
   &\leq\frac{1}{2}B\left(n-\frac{(2n-1)p}{2(p-1)}, \frac{n}{p-1}\right)
\left(\int_{\mathbb{B}_n}\vert\nabla_g\ u\vert^p_g dv_g\right)^{\frac{1}{p}}.
    \end{aligned}
\end{equation*}
\end{proof}

\section{Monotonicity for the Superlevel Sets}\label{section3}

In this section, we always suppose the Assumption \ref{assumption} is satisfied. We will prove a general monotonicity theorem for the hyperbolic measure of the superlevel sets of holomorphic functions which used an ingenious method from \cite{NT22} to the complex hyperbolic ball. 

Let $f$ be a holomorphic function in  $\mathbb{B}_n$ such that $u(z)=\vert f(z)\vert^{a}(1-\vert z\vert^2)^b$ is bounded and goes to $0$
uniformly as $ z\rightarrow \mathbb{S}_n$.
Then the superlevel sets
$$A_t:=\{z\in\mathbb{B}_n: u(z)>t\}$$
for $t>0$ are relatively compact subsets of  $\mathbb{B}_n$, and  thus have finite hyperbolic measure
$$\mu(t)=v_g(A_t).$$
The goal of this section is to prove the following theorem which says
that a certain quantity related to this measure is decreasing.
\begin{theorem}\label{mon}
With Assumption \ref{assumption},
let $f:\mathbb{B}_n\rightarrow\mathbb{C}$ be a holomorphic function such that the function
$$u(z)=\vert f(z)\vert^{a}(1-\vert z\vert^2)^b$$ is bounded, and  tends to $0$
uniformly as $ z\rightarrow \mathbb{S}_n$.
Then the function $$g(t)=t^{\frac{1}{b}}(\mu^{\frac{1}{n}}+1)$$ is decreasing on the interval $(0,t_0)$, where
$t_0=\max_{z\in\mathbb{B}_n}u(z)$.
\end{theorem}

\begin{proof}
By Coarea formula, for every Borel function $h:\mathbb{B}_n\rightarrow [0,+\infty)$, we have
\begin{equation*}
\int_{\mathbb{B}_n}h(z)\vert\nabla_g u\vert_g d v_g(z)=\int_0^{\max u}\left(\int_{u(z)=\kappa} h d\sigma_g\right)d\kappa,
\end{equation*}
where $d\sigma_g$ denotes the hyperbolic surface area measure on  $\partial A_\kappa=\{z\in\mathbb{B}_n: u(z)=\kappa\}$ induced by the Bergman metric on $\mathbb B_n$.

Set
$$h(z)=\chi_{A_t}(z)\vert\nabla_g u\vert_g^{-1}, $$
and then for a.e. $z$ we have
$$h(z)\vert\nabla_g u\vert_g=\chi_{A_t}(z).$$

Hence, it follows that
\begin{equation*}
\mu(t)=\int_{\mathbb{B}_n}\chi_{A_t}(z)d v_g(z)=\int_t^{\max u}\left(\int_{u(z)=\kappa}\vert\nabla_g u\vert_g^{-1} d\sigma_g\right)d\kappa.
\end{equation*}

Thus, we can deduce that
\begin{equation*}
-\mu'(t)=\int_{u(z)=t}\vert\nabla_g u\vert_g^{-1} d\sigma_g.
\end{equation*}

Our next step is to apply the Cauchy-Schwarz inequality to the hyperbolic surface area  of $\partial A_t$:
\begin{equation*}
[{\rm Area}_g(\partial A_t)]^2=\left(\int_{\partial A_t} d\sigma_g\right)^2\leq\left(\int_{\partial A_t}\vert\nabla_g u\vert_g^{-1} d\sigma_g\right)
\left(\int_{\partial A_t}\vert\nabla_g u\vert_g d\sigma_g\right).
\end{equation*}

Let $\nu$ be the outward normal to $\partial A_t$ with respect to the Bergman metric. Then we have
$\vert\nabla_g u\vert_g =-\langle \nabla_g u, \nu\rangle_g$. Thus, for $z\in\partial A_t$, we have
$$\frac{\vert\nabla_g u\vert_g }{t}=\frac{\vert\nabla_g u\vert_g }{u(z)}=-\frac{\langle \nabla_g u, \nu\rangle_g}{u}=-\langle \nabla_g\log u, \nu\rangle_g.$$

Then, we deduce that
\begin{equation*}
\begin{aligned}
[{\rm Area}_g(\partial A_t)]^2&\leq t\mu'(t)\int_{\partial A_t}\langle \nabla_g\log u, \nu\rangle_g d\sigma_g\\
&=t\mu'(t)\int_{A_t}\dive(\nabla_g\log u)d v_g\\
&=t\mu'(t)\int_{A_t}\triangle_g(\log u) dv_g.
\end{aligned}
\end{equation*}
Here $\triangle_g$ is the invariant Laplacian operator on $\mathbb{B}_n$, given by
$$\triangle_g=4(1-\vert z\vert^2)\sum_{i,j=1}^n(\delta_{ij}-z_i\bar z_j)\frac{\partial^2}{\partial z_i\partial \bar z_j}.$$

Note that $u(z)\neq 0$ for $z\in A_t$, and this means that $\log u(z)$ is well-defined on $A_t$ and $\partial A_t$, and directly computation shows that
$\triangle_g\log u(z)=-4nb$. Thus, we have
$$ [{\rm Area}_g(\partial A_t)]^2\leq -4nbt\mu(t)\mu'(t).$$
Together with the isoperimetric inequality established in Theorem \ref{Iso}, we obtain that for $t>0$
$$\frac{bt}{n}\mu'(t)\mu^{\frac{1}{n}-1}(t)+\mu^{\frac{1}{n}}(t)+1\leq 0.$$
By a direct calculation, $$g'(t)=\frac{t^{\frac{1}{b}-1}}{b}\left(\frac{bt}{n}\mu'(t)\mu^{\frac{1}{n}-1}(t)+\mu^{\frac{1}{n}}(t)+1\right)\leq 0.$$
Thus, $g(t)$ is decreasing.
\end{proof}
\section{Weak-Type Estimate for the Hardy Spaces}
In this section we are going to prove the existence of a  upper bound for the hyperbolic measure of the superlevel sets of functions in the Hardy spaces.

\begin{theorem}\label{bound}
With assumption \ref{assumption},
let $f\in H^{nr}$ with  $\|f\|_{H^{nr}}=1$. Set $u(z)=\vert f(z)\vert^r(1-\vert z\vert^2)$ and its distribution function $\mu(t)={\rm vol}_g(\{z\in\mathbb B_n: u(z)>t\})$. Then for all $t\in (0,\infty)$ we have
\begin{equation}\label{upper bound}
\mu(t)\leq\max\left\{\left(\frac{1}{t}-1\right)^n, 0\right\}.
\end{equation}
Equality in \eqref{upper bound} holds for all $0<t<\infty$ if $f(z)\equiv 1$.
\end{theorem}

\begin{proof}
Put $t_0=\max_{z\in\mathbb{B}_n}u(z)$. By the pointwise bound \eqref{hardy} we have $t_0\leq 1$. Clearly, for $t\geq t_0$, \eqref{upper bound} is trivial.

Assume that here exists some $0<t_1<t_0$ such that $\mu(t_1)^{\frac{1}{n}}>\frac{1}{t_1}-1$. Then $\mu(t_1)^{\frac{1}{n}}=\frac{c}{t_1}-1$ for some $c>1$.

Claim: for all $0<t<t_1$ we always have
$$\mu(t)^{\frac{1}{n}}\geq\frac{c}{t}-1.$$
In fact, by Theorem \ref{mon} with $a=r$ and $b=1$, we deduce that $g(t)=t\left(\mu(t)^{\frac{1}{n}}+1\right)$ is decreasing. Since $g(t_1)=c$, we obtain
$g(t)\geq c$ for $0<t<t_1$. Thus, it follows that $$\mu(t)^{\frac{1}{n}}\geq\frac{c}{t}-1.$$
Hence, we have
\begin{equation*}
\begin{aligned}
\Vert f\Vert_{A_\alpha^{r\alpha}}^{r\alpha}
&=\int_{\mathbb{B}_n}c_\alpha\vert f(z)\vert^{r\alpha}(1-\vert z\vert^2)^\alpha d v_g(z)\\
&=\int_0^{t_0}c_\alpha dt\int_{\partial A_t}u(z)^\alpha \vert\nabla_g u\vert_g^{-1} d\sigma_g=\int_0^{t_0}-c_\alpha t^\alpha \mu'(t) dt\\
&\geq \alpha c_\alpha \int_0^{t_1} \mu(t) t^{\alpha-1} dt\geq \alpha c_\alpha \int_0^{t_1} \left(\frac{c}{t}-1\right)^n t^{\alpha-1} dt.
\end{aligned}
\end{equation*}
Finally, we have
\begin{equation}\label{1}
\Vert f\Vert_{A_\alpha^{\alpha r}}^{\alpha r}\geq k_\alpha \int_0^{t_1} \left(\frac{c}{t}-1\right)^n t^{\alpha-1} dt,
\end{equation}
where $k_\alpha=(\alpha-n)\frac{\Gamma(\alpha+1)}{\Gamma(n+1)\Gamma(\alpha-n+1)}$.

On the other hand, we have
\begin{equation*}
\begin{aligned}
1&=k_\alpha\int_0^1 \left(\frac{1}{t}-1\right)^n t^{\alpha-1} dt\\
&=k_\alpha\int_0^{t_1} \left(\frac{1}{t}-1\right)^n t^{\alpha-1} dt+k_\alpha\int_{t_1}^1 \left(\frac{1}{t}-1\right)^n t^{\alpha-1} dt\\
&:=A(\alpha)+B(\alpha).
\end{aligned}
\end{equation*}

Since $k_\alpha\rightarrow 0$ as $\alpha\rightarrow n$, we see that $B(\alpha)\rightarrow 0$ as $\alpha\rightarrow n$. Hence, we deduce that
$A(\alpha)\rightarrow 1$ as $\alpha\rightarrow n$. On the other hand, we have
$$\frac{\frac{c}{t}-1}{\frac{1}{t}-1}=c+\frac{c-1}{\frac{1}{t}-1}\geq c.$$
Thus, together with \eqref{1}, we have that
\begin{equation*}
1=\Vert f\Vert_{H^{nr}}^{nr}=\lim_{\alpha\rightarrow n}\Vert f\Vert_{A_\alpha^{\alpha r}}^{\alpha r}\geq c^n\lim_{\alpha\rightarrow n} A(\alpha)=c^n>1,
\end{equation*}
which leads to a contradiction.
\end{proof}

\begin{proof}[Proof of Theorem \ref{increasing}]
Write $u(z)=|f(z)|^r(1-|z|^2)$ for $f\in H^{nr}$ with $\|f\|_{H^{nr}}=1$ and $\mu(t)={\rm vol}_g(\{z\in\mathbb B_n: u(z)>t\})$.
We shall assume that $\lim_{t\rightarrow 0^+}G(t)=0$ just as in the proof of Kulikov \cite{Kulikov}. Then this integral can be expressed via $\mu(t)$ as
\begin{equation*}
\begin{aligned}
\int_{\mathbb{B}_n}G\left(\vert f\vert^r(1-\vert z\vert^2)\right) dv_g(z)&=\int_0^{\infty}dt\int_{u(z)=t}G\left(u(z)\right)\vert\nabla_g u(z)\vert_g^{-1}d\sigma_g\\
%&=\int_0^{+\infty} G(t)(-\mu('(t))dt\\
&=\int_0^{+\infty} \mu(t) dG(t)\\
&\leq \int_0^{+\infty}\max\left\{\left(\frac{1}{t}-1\right)^n,0\right\} dG(t).
\end{aligned}
\end{equation*}
This last inequality is due to Theorem \ref{bound} and $dG(t)$ is a positive measure, as $G(t)$ is increasing. Also, we see that the maximal value for the integral will be taken for
$f(z)\equiv 1$.
\end{proof}

\section{Weighted Bergman space, Proof of Theorem \ref{convex}}

Assume that $\lim_{t\rightarrow 0^+}G(t)=0$, and the integral in \eqref{convex case} can be rewritten as
\begin{equation}\label{5.1}
\int_0^\infty \mu(t) G'(t)d t,
\end{equation}
where $\mu(t)=v_g(\{z\in\mathbb{B}_n: u(z)>t\})$ and $u(z)=\vert f(z)\vert^p(1-\vert z\vert^2)^\alpha$.
We also assume that $\Vert f\Vert_{A_\alpha^p}=1$, that is,
\begin{equation}\label{5.2}
\int_0^\infty \mu(t) dt=\frac{1}{c_\alpha}=\frac{n!\Gamma(\alpha-n)}{\Gamma(\alpha)}.
\end{equation}

Applying Theorem \ref{mon} to $f$ with $a=p$ and $b=\alpha$, we have
\begin{equation*}
\mu(t)^{\frac{1}{n}}=\frac{g(t)}{t^{1/\alpha}}-1
\end{equation*}
where $g$ is decreasing on $(0,t_0)$ with $t_0=\max_{z\in\mathbb{B}_n} u(z)$. By the pointwise bound \eqref{berg} we deduce
that $t_0\leq 1$. 

In order to maximize  \eqref{5.1} under the condition \eqref{5.2}, we need the following lemma due to Kalaj \cite{Kalaj}. 
\begin{lemma}\label{lem5.1}
Assume that $\Phi$, $\Psi$ are positive increasing functions and $g$ is a positive nonincreasing function such that
$$\int_0^{1}\Phi\left(\frac{g(t)}{t^{1/\alpha}}\right)dt=\int_0^{1}\Phi\left(\frac{1}{t^{1/\alpha}}\right)dt=c.$$
Then
$$\int_0^{1}\Phi\left(\frac{g(t)}{t^{1/\alpha}}\right)\Psi(t)dt\leq\int_0^{1}\Phi\left(\frac{1}{t^{1/\alpha}}\right)\Psi(t)dt.$$
\end{lemma}
\begin{proof}
 This is merely a variation of Lemma 4.1 in Kalaj \cite{Kalaj}, but for the convenience of the readers, we will provide a complete proof here.  Choose $a\in[0,1]$ such that $g(t)\geq 1$ for $t\leq a$ and $g(t)\leq 1$ for $t\geq a$. Then 
 $$h(t):=\left(\Phi\left(\frac{g(t)}{t^{1/\alpha}}\right)-\Phi\left(\frac{1}{t^{1/\alpha}}\right)\right)\left(\Psi(t)-\Psi(a)\right)\leq 0$$
for all $t\in[0,1]$. By integrating $h(t)$ for $t\in(0, 1)$ we get
\begin{equation*}
 \begin{aligned}
 &\int_0^{1}\Phi\left(\frac{1}{t^{1/\alpha}}\right)
 \Psi(a)-\Phi\left(\frac{g(t)}{t^{1/\alpha}}\right)\Psi(a)-\Phi\left(\frac{1}{t^{1/\alpha}}\right)\Psi(t)+\Phi\left(\frac{g(t)}{t^{1/\alpha}}\right)\Psi(t)dt\\
 =&\Psi(a)\int_0^{1}\left(\Phi\left(\frac{g(t)}{t^{1/\alpha}}\right)-\Phi\left(\frac{1}{t^{1/\alpha}}\right)\right)dt+\int_0^{1}\Psi(t)\left(\Phi\left(\frac{g(t)}{t^{1/\alpha}}\right)-\Phi\left(\frac{1}{t^{1/\alpha}}\right)\right)dt
 \leq 0.
 \end{aligned} 
\end{equation*}
 Since 
 $$\int_0^{1}\left(\Phi\left(\frac{g(t)}{t^{1/\alpha}}\right)-\Phi\left(\frac{1}{t^{1/\alpha}}\right)\right)dt=0,$$
 it follows that 
 $$\int_0^{1}\Phi\left(\frac{g(t)}{t^{1/\alpha}}\right)\Psi(t)dt\leq\int_0^{1}\Phi\left(\frac{1}{t^{1/\alpha}}\right)\Psi(t)dt.$$
\end{proof}
Now let $\Phi(x)=(x-1)^n$ with $x\geq 1$, and $\Psi(t)=G'(t)$. Thus, we have that the maximum of \eqref{5.1} under the condition \eqref{5.2} is attained when $g(t) \equiv 1$ , which implies that $f \equiv 1 $. Therefore, we finish the proof of Theorem \ref{convex}.

\bigskip

\begin{center}Acknowledgments\end{center}
The authors thanks Prof. Congwen Liu for providing the reference \cite{Mc03} and his interest in this work. The first author was partially supported by the National Natural Science Foundation of China (12361131577, 12271411). The second author was partially supported by the National Natural Science Foundation of China (12361131577).

%%%%%%%%%%%%%%%%%%%%%%%%%%%%%%%%%%%%%%%%%%%%%%%%%%%%%%%%%%%%%%%%

%%%%%%%%%%%%%%%%%%%%%%%%%%%%%%%%%%%%%%%%%%%%%%%%%%%%%%%%%%%%%%%%
\addcontentsline{toc}{section}{References}
\phantomsection
\renewcommand\refname{References}

%%%%%%%%%%%%%%%%%%%%%%%%%%%%%%%%%%%%%%%%%%%%%%%%%%%%%%%%%%%%%%%%
\end{document}